\documentclass[11pt,reqno]{amsart}

\usepackage{setspace}
\usepackage{color}
\usepackage[all]{xy}
\usepackage{amssymb, amsmath, amsfonts, amsthm}
\usepackage{enumitem}
\usepackage[unicode=true]{hyperref}
\usepackage{comment}

\calclayout

\usepackage{mathabx}
\usepackage{amsfonts}
\usepackage{pst-node}
\usepackage{tikz-cd} 
\usepackage{graphicx}
\usepackage{pifont}
\usepackage[T1]{fontenc}
\usepackage{color}
\usepackage[all]{xy}
\usepackage{amssymb, amsmath, amsfonts, amsthm}
\usepackage{enumitem}

\newtheorem{theorem}{Theorem}[section]
\newtheorem{proposition}[theorem]{Proposition}
\newtheorem{corollary}[theorem]{Corollary}
\newtheorem{lemma}[theorem]{Lemma}
\theoremstyle{definition}
\newtheorem{definition}[theorem]{Definition}
\newtheorem{example}[theorem]{Example}

\newtheorem{remark}[theorem]{Remark}

\newcommand{\bbZ}{\mathbb{Z}}
\newcommand{\cC}{\mathcal{C}}
\newcommand{\cD}{\mathcal{D}}

\newcommand{\cG}{\mathcal{G}}
\newcommand{\bP}{\mathbb{P}}
\newcommand{\bQ}{\mathbb{Q}}

\DeclareMathOperator*{\hocolim}{hocolim}

\newcommand{\Hom}{\mathrm{Hom}}
\newcommand{\id}{\mathrm{id}}
\newcommand{\diag}{\mathrm{diag}}
\newcommand{\Ext}{\mathrm{Ext}}
\newcommand{\Tot}{\mathrm{Tot}}

\def\maprt#1{\smash{\,\mathop{\longrightarrow}\limits^{#1}\,}}

\makeatletter
\newcommand{\addresseshere}{%
  \enddoc@text\let\enddoc@text\relax
}
\makeatother

\begin{document}
 
\title{Thomason Cohomology and Quillen's Theorem~A}
\author{Mehmet Kırtışoğlu}
\author{Ergün Yalçın$^{\dagger}$}
 
\date{}
 

\thanks{\emph{${}^\dagger$Corresponding Author:} Erg\" un Yal\c c\i n, Email: yalcine@fen.bilkent.edu.tr \\
Address: Bilkent University, Department of Mathematics, 06800, Bilkent, Ankara, T\" urkiye}
\thanks{\emph{Contributing Author:} Mehmet K\i rt\i \c so\u glu,  Email: m.kirtisoglu@bilkent.edu.tr \\
Address: Bilkent University, Department of Mathematics, 06800, Bilkent, Ankara, T\" urkiye}

\begin{abstract}
Given a functor $\varphi : \cC \to \cD$ between small categories, Quillen showed that there is a homotopy equivalence $\kappa:\hocolim_{\cD} N(\varphi /-) \to N\cC$, where $N(\varphi/-)$ is the functor that sends each object $d$ of $\cD$ to the nerve of the comma category $\varphi/d$.  
We show that the homotopy equivalence $\kappa$ induces an isomorphism on the Gabriel-Zisman cohomology of the associated simplicial sets. 
As a consequence, we obtain a version of Quillen's Theorem A for the Thomason cohomology of categories. We also construct a spectral sequence converging to the Thomason cohomology of the Grothendieck construction $\int_{\cD} F$ of a functor $F:\cD \to \mathbf{Cat}$ with arbitrary coefficients.
\end{abstract}

\maketitle
\setcounter{tocdepth}{1}
\tableofcontents


\section{Introduction and Statements of Results}\label{sect:Introduction}

Let $\cC$ be a small category. Given a coefficient system $M$ on its nerve $N\cC$, the Gabriel--Zisman cohomology of the simplicial set $N\cC$ is called the
\emph{Thomason cohomology} of $\cC$ with coefficients in $M$, and is denoted by
$H^{\bullet} _{\mathrm{Th}}(\cC;M)$. This theory was introduced by G\' alvez-Carrillo, Neumann, and Tonks in \cite{GNT-Thomason-2013} and it includes, as special cases, the ordinary cohomology of small categories
and the Baues--Wirsching cohomology of categories with coefficients in natural systems.

For a functor $\varphi \colon \cC \to \cD$, Quillen \cite{Quillen-KTheory-1973} 
proved that there is a natural weak
homotopy equivalence
\[
\kappa \colon \operatorname{hocolim}_{\cD} N(\varphi/{-}) \longrightarrow N\cC,
\]
where $\varphi/{-}$ is the functor that assigns to each object $d$ of $\cD$ the
comma category $\varphi/d$.
This result is a key ingredient in Quillen’s Theorem~A. Here the homotopy colimit $\hocolim_{\cD} N(\varphi/-)$ is the diagonal of the simplicial replacement of the functor $N(\varphi/-)$.

In general a weak homotopy equivalence between two simplicial sets does not induce an isomorphism between their Gabriel-Zisman cohomology groups with arbitrary coefficients. In this paper we show that the weak homotopy equivalence $\kappa$ induces an isomorphism between the corresponding Gabriel-Zisman cohomology groups. The main theorem of the paper is the following:
\begin{theorem}\label{thm:Main1}
Let $\varphi : \cC \to \cD$ be a functor between two small categories. For every coefficient system $M \colon \Delta(N \cC) \to R\text{-Mod}$, the simplicial map $\kappa$ induces an isomorphism
\[
\kappa ^\ast \colon H^\bullet_{Th}(\cC; M)
\;\cong\;
H^\bullet\bigl(\mathrm{hocolim}_{\cD} N(\varphi/{-}); \kappa^\ast M\bigr).
\]
\end{theorem}

Theorem \ref{thm:Main1} is proved by constructing a double complex of projective $R\Delta (N\cC)$-modules whose total complex gives a projective resolution for the constant functor $\underline R$ as an $R\Delta (N\cC)$-module. This idea of constructing a projective resolution as a double complex was also used by Cegarra in the proof of \cite[Theorem 1]{Cegarra-2020}. As a special case of Theorem \ref{thm:Main1}, we obtain that the simplicial map $\kappa$ induces isomorphisms for the Baues-Wirsching cohomology and the ordinary cohomology of $\cC$.

As an application of Theorem \ref{thm:Main1}, we prove a version of Quillen's Theorem A for Thomason cohomology. For every object $d$ of $\cD$, let  $u : \varphi/d \to \id_{\cD}/d$ be the functor that sends each object $(c, \mu)$ of $\varphi/d$ to $(\varphi (c), \mu)$ in $\id_{\cD} /d$, and $j: \varphi/d \to \cC$ be the forgetful functor that sends each object $(c, \mu)$ to  $c$. For every object $d$ of $\cD$, there is a commuting diagram of functors
\[
\xymatrix{ \varphi/d \ar[r]^{j} \ar[d]_{u}  & \cC \ar[d]^{\varphi} \\
\id _{\cD}/d \ar[r]^{j} & \cD.}
\]
 
\begin{theorem}\label{thm:Main2-QuillenA}
Let $\varphi : \cC \to \cD$ be a functor and $M$ be any coefficient system for $N\cD$. Suppose that
for each object $d$ of $\cD$, the homomorphism 
\[
u^* : H^* _{Th} ( \id_{\cD}/d ; j ^* M  )\cong H^*_{Th}  ( \varphi/d ; j ^* \varphi^*M)
\]
is an isomorphism. Then the functor $\varphi$ induces an isomorphism 
\[
\varphi ^* : H^* _{Th} (\cD ; M) \to H^* _{Th} (\cC; \varphi ^* M ).
\]
\end{theorem}

Proof of Theorem \ref{thm:Main2-QuillenA} follows from a spectral sequence that we construct for every functor $\varphi : \cC \to \cD$ which converges to the Thomason cohomology $H^\bullet _{Th} (\cC; M)$ and whose $E_2$-term can be expressed in terms of the cohomology of $\cD$. As a special case of this spectral sequence we obtain a spectral sequence for the Thomason cohomology of the Grothendieck construction $\int _{\cD} F$ of a functor $F : \cD \to \mathbf{Cat}$.

Let $F: \cD\to \mathbf{Cat}$ be a functor and $\int _{\cD} F$ denote the Grothendieck construction of $F$. The objects of $\int_{\cD} F$ are the pairs $(d,x)$ where $d$ is an object of $\cD$ and $x$ is an object of $F(d)$. There is a forgetful functor
\[
\pi:\int_{\cD} F\to \cD
\]
that sends every object $(d,x)$ of $\int_{\cD} F$ to $d$. For every object $d$ in $\cD$, let $j: \pi/d \to \int _{\cD} F$ be the functor that sends a pair $((d', x'), \mu)$ to $(d',x')$ in $\int_{\cD} F$.

\begin{theorem}\label{thm:GrothSpectSeq} 
For every coefficient system $M:\Delta\bigl(N\!\int_D F\bigr)\to R\text{-}\mathrm{Mod}$, there is a first quadrant
spectral sequence
\[
E^{p,q}_2
=
H^p\bigl(\cD^{op}; H^q_{\mathrm{Th}}(\pi/-; j^*M)\bigr)
\;\Longrightarrow\;
H^{p+q}_{Th}\!\left(\int_{\cD} F;M\right).
\]
\end{theorem}

A similar spectral sequence that converges to the Baues-Wirsching cohomology of the Grothendieck construction $\int_{\cD} F$ was constructed earlier by Pirashvili and Redondo \cite{Pirashvili-Redondo-2006} and by Cegarra \cite{Cegarra-2020}. Theorem \ref{thm:GrothSpectSeq} can be considered as a generalization of these spectral sequences constructed by Pirashvili-Redondo and Cegarra.

\vskip 5 pt

\emph{Notation and Conventions:} All the categories are assumed to be small (or equivalent to a small category). 
Throughout the paper $R$ denotes a commutative ring with unity.


\section{Cohomology of Small Categories}\label{sect:Prelim}

 This is a preliminary section where we introduce necessary definitions on the cohomology of categories and the Gabriel-Zisman cohomology of simplicial sets.  For more details on this material we refer the reader to 
\cite{Quillen-KTheory-1973}, \cite{Yalcin-2024}, \cite{Gabriel-Zisman-Book-1967}, \cite{Richter-Book-2020},
and \cite{GNT-Thomason-2013}.


\subsection{Ordinary Cohomology of Categories}\label{sect:CohCat}

Let $\cC$ be a small category and $R$ a commutative ring.
A functor $M: \cC \to R$-Mod is called a (covariant) \emph{$R\cC$-module}. The category of $R\cC$-modules is an abelian category where the exactness of a sequence $M_1\to M_2 \to M_3$ in $R\cC$-Mod is defined by the component-wise exactness of the corresponding $R$-modules $M_1(c) \to M_2(c) \to M_3(c)$ for every object $c$ of $\cC$. 

There is an explicit construction of projective $R\cC$-modules 
using the Yoneda Lemma:   For every object $x$ of $\cC$, let $R\Hom _{\cC} (x, -)$ denote the functor $\cC \to R$-Mod that sends an object $c$ of $\cC$ to the free $R$-module with basis $\Hom _{\cC} (x, c)$. By the Yoneda Lemma,
for every $R\cC$-module $M$, 
\[
\Hom _{R\cC} (R\Hom _{\cC} (x, -), M ) \cong M(x).
\]
Using this isomorphism, one can show that for every object $x$ of $\cC$, the $R\cC$-module $R \Hom_{\cC} (x, -)$ is projective
(see \cite{Lueck-Book-1989}, \cite{Webb-Survey-2007} for details). Having enough projective modules allows us to construct projective resolutions and define ext-groups over $R\cC$-modules.
 
For a small category $\cC$ and a commutative ring $R$, the constant functor $\underline R$ is the $R\cC$-module 
such that for every object $c$ of $\cC$, $\underline R (c)=R$, and for every morphism $\alpha : c\to c'$ in $\cC$, the induced map $\underline R (\alpha) : R \to R$ is the identity map.

\begin{definition} The \emph{(ordinary) cohomology of a small category} $\cC$ with coefficients in an $R\cC$-module $M$ is defined by 
\[
H^* (\cC; M) := \Ext ^* _{R\cC} (\underline R , M).
\]
\end{definition}

One can also define the ordinary cohomology of a small category $\cC$ as a special case of the Gabriel-Zisman cohomology of its nerve $N\cC$ (see Section \ref{sect:ThCoh}).


\subsection{Gabriel-Zisman Cohomology of Simplicial Sets}\label{sect:CohSimpSet}

The \emph{simplex category} $\Delta$ is the category whose objects are ordered finite sets $[n]=\{ 0, 1, \dots, n\}$, and whose morphisms are given by weakly order-preserving functions between them. Morphisms in the simplex category are generated by coboundary maps $d^i : [n-1] \to [n]$, $0\leq i\leq n$, and codegeneracy maps $s^i: [n+1] \to [n]$,  $0\leq i \leq n$, which satisfy certain relations (see \cite{Goerss-Jardine-Book-2009} for details). 

A functor $X: \Delta ^{op} \to \cC$ is called a simplicial object in $\cC$. 
If $\cC$ is the category of $R$-modules, then $X$ is called a simplicial $R$-module. The homology of a simplicial $R$-module is defined as the cohomology of its Moore complex. The cohomology of a cosimplicial $R$-module is defined as the cohomology of the associated Moore cochain complex.

A simplicial object in Sets is called a simplicial set. When $X$ is a simplicial set, for each $n\geq 0$, we denote $X[n]$ by $X_n$, and call it the set of $n$-simplices. For every morphism $f:[m]\to [n]$ in $\Delta$, we denote $X(f): X_n\to X_m$ by $f^*$. For each $i\in [n]$, the face maps and degeneracy maps are denoted by $d_i: X_n \to X_{n-1}$ and $s_i: X_n \to X_{n+1}$.

\begin{definition}\label{def:CoeffSystem}  Given a simplicial set $X$, the \emph{category of simplices of $X$} is the category $\Delta(X)$ whose objects are the pairs $([n], \sigma)$ where $\sigma \in X_n$, and the morphisms $f: ([m], \tau)\to ([n], \sigma)$ in $\Delta(X)$ are given by the morphisms $f: [m]\to [n]$ in $\Delta$ such that $f^* \sigma =\tau$. A functor $M: \Delta (X) \to R$-Mod is called a \emph{coefficient system} on $X$.
\end{definition}

We write $M (\sigma )$ for $M ([n], \sigma)$, and for each morphism $f: ([m], \tau) \to ([n], \sigma)$, the induced map $M (\tau) \to M (\sigma)$ is denoted by $M(f)$. In particular, for each $i$, the coboundary and codegeneracy maps induce $R$-module homomorphisms 
\[
M(d^i): M (d_i \sigma ) \to M (\sigma) \quad \text{ and } \quad M(s^i)  : M (s_i \sigma ) \to M (\sigma).
\]
A coefficient system is called a \emph{local coefficient system} if for every morphism $f: ([m], \tau) \to ([n], \sigma)$ in $\Delta (X)$, the induced homomorphism $M(f): M (\tau) \to M (\sigma)$ is an isomorphism.

\begin{definition}\label{def:GabrielZisman} Let $X$ be a simplicial set and $M$ a coefficient system on $X$. Consider the cosimplicial $R$-module $C (X; M)$ such that for each $n\geq 0$, 
\[
C^n (X; M )=\{ \zeta : X_n \to \prod _{\sigma \in X_n} M(\sigma)  \, | \, \zeta (\sigma ) \in M (\sigma) \} 
\]
where for each $f: [m]\to [n]$, the induced map $f^*: C^m(X; M) \to C^n (X; M)$ is defined by 
\[
f^* (\zeta ) (\sigma)= M(f) \bigl ( \zeta ( f^* \sigma ) \bigr )
\]
for every $\zeta \in C^m (X; M)$ and $\sigma \in X_n$. The \emph{Gabriel-Zisman cohomology} $H^\bullet (X; M)$ is defined to be the cohomology of the cosimplicial module $C (X; M)$.
\end{definition}

Alternatively, one can define the cohomology group $H^\bullet (X; M)$ as the ordinary cohomology of the simplex category $\Delta (X)$ with coefficients in the $R\Delta (X)$-module $M$. These two definitions give isomorphic cohomology groups (see Gabriel–Zisman \cite[Appendix II]{Gabriel-Zisman-Book-1967}).


\subsection{Induced Maps on Cohomology}

Let $\lambda: X\to Y$ be a simplicial map between two simplicial sets. The map $\lambda$ induces a functor $\Delta(\lambda): \Delta (X)\to \Delta (Y)$ between corresponding simplex categories. For every coefficient system $M : \Delta (Y) \to R$-Mod on $Y$, the composition 
\[ 
\lambda ^* M : \Delta (X) \maprt{\Delta (\lambda)} \Delta (Y) \maprt{M } R\text{-Mod}
\]
is a coefficient system on $X$.

\begin{lemma}\label{lem:InducedHom}
The simplicial map $\lambda: X \to Y$ induces an $R$-module 
homomorphism 
\[
\lambda ^* : H^\bullet (Y; M ) \to H^\bullet (X; \lambda^* M )
\]
for every coefficient system $M$ on $Y$.
\end{lemma}

\begin{proof}
The simplicial map $\lambda$ induces a chain map 
\[
\lambda ^* : C^n ( Y, M ) \to C^n (X; \lambda ^* M )
\]
defined by $\lambda ^* (\zeta) (\sigma) = \zeta (\lambda (\sigma) )$
for every simplex $\sigma \in X_n$. Note that 
\[
\lambda ^* (\zeta) (\sigma )=\zeta (\lambda (\sigma)) \in M (\lambda (\sigma) )= (\lambda ^* M ) ( \sigma).
\]
The chain map $\lambda^*$ induces the desired homomorphism between the corresponding cohomology modules.
\end{proof}


\subsection{The Thomason Cohomology of a Category}\label{sect:ThCoh}

The nerve $N\cC$ of the category $\cC$ is a simplicial set whose $n$-simplices $N\cC_n$ are given by the length $n$ chains of composable morphisms $\sigma = \bigl (\sigma(0) \maprt{\sigma_1} \cdots \maprt{\sigma_n} \sigma(n) \bigr )$ in $\cC$. For $n=0$, we take $NC_0$ to be the set of objects in $\cC$. We can consider an $n$-simplex in $N\cC$ as a functor $\sigma : [n]\to \cC$. Then for every $f:[m]\to [n]$, the induced map $f^* : N\cC_n \to N\cC_m$ is defined by $f^* (\sigma) =\sigma f$.

The face and degeneracy maps for $N\cC$ can be described explicitly as follows: For $n \geq 1$, the face map $d_i : N\cC_n \to N \cC_{n-1}$, $0 \leq i\leq n$, is defined by 
\[
d_i ( \sigma(0) \maprt{\sigma_1} \cdots \maprt{\sigma_n} \sigma(n))=\begin{cases} \sigma(1) \maprt{\sigma_2} \cdots \maprt{\sigma_n} \sigma(n) & \text{ if } i=0 \\
\sigma(0) \maprt{\sigma_1} \cdots \sigma(i-1) \maprt{\sigma _{i+1} \sigma_i } \sigma(i+1) \cdots \maprt{\sigma_n} \sigma(n)  & \text{ if } 0< i< n \\
\sigma(0) \maprt{\sigma_1} \cdots \maprt{\sigma_{n-1}} \sigma(n-1) & \text{ if } i=n. 
\end{cases}
\]
The degeneracy map $s_i : N\cC _{n} \to N\cC _{n+1}$, $0 \leq i\leq n$, is defined by 
\[
s_i ( \sigma(0) \maprt{\sigma_1} \cdots \maprt{\sigma_n} \sigma(n))= \sigma(0) \maprt{\sigma_1} \cdots \sigma(i) \maprt{\id} \sigma(i) \cdots  \maprt{\sigma_n} \sigma(n).
\]

\begin{definition}\label{def:ThomasonCoh} Let $\cC$ be a small category.  
Given a coefficient system  $M : \Delta (N\cC)\to R$-mod on the nerve $N\cC$,  the \emph{Thomason cohomology} of $\cC$ with coefficients in $M$ is defined by 
\[
H^\bullet _{Th} (\cC; M ):=H^\bullet (N\cC; M ).
\]
\end{definition}

By the definition of the Gabriel-Zisman cohomology, the Thomason cohomology $H_{Th} ^{\bullet} (\cC; M)$ is the cohomology of the cosimplicial $R$-module $C(\cC; M)$ where for every $n\geq 0$,  
\begin{equation}\label{eqn:CosimplicialModuleForC}
C^n(\cC; M) = \prod_{\sigma \in N\cC_n} M(\sigma)
\end{equation}
together with the face and degeneracy maps defined as in Definition 
\ref{def:GabrielZisman}.

The Thomason cohomology recovers the ordinary cohomology of a small category $\cC$ in the following sense: Let $l:\Delta(N\cC) \to \cC$ be the last-vertex functor sending $\sigma \in N\cC_n$ to $\sigma(n)$. Then for every $R\cC$-module $M: \cC \to R\text{-Mod}$, there is an isomorphism:
\[
H^\bullet (\cC; M) \cong H_{Th}^\bullet (\cC; l^*M).
\]

The Thomason cohomology also recovers the Baues-Wirsching cohomology of a category. For a small category $\cC$, the factorization category $F\cC$ is the category whose objects are the morphisms $\alpha : x\to y$ in $\cC$, and whose morphisms $(u, v): \alpha \to \alpha'$  are given by a pair of morphisms $(u, v)$ in $\cC$ such that the following diagram commutes: 
\[
\xymatrix{x  \ar[r]^{\alpha}  & y \ar[d]^u\\ x' \ar[u]^v \ar[r]^{\alpha'}  & y'}
\]
A functor $M: F\cC \to R$-Mod is called a \emph{natural system} for $\cC$ over $R$. The Baues-Wirsching cohomology $H^\bullet _{BW} (\cC; M)$ of a small category $\cC$ with coefficients in a natural system $M$ is defined as the ext-group $\Ext ^\bullet _{RF\cC} ( \underline R, M)$. It can be also defined as the cohomology of an explicit cochain complex (see \cite{Baues-Wirsching-1985}).

There is a functor $\xi: \Delta (NC) \to F\cC$ that takes an $n$-simplex $\sigma =(\sigma (0) \maprt{\sigma_1} \cdots \maprt{\sigma_n}  \sigma (n))$ in $N\cC$ to the morphism $\sigma_n \cdots \sigma _1: \sigma(0) \to \sigma(n)$ in $\cC$. 
For every natural system $M: F\cC \to R\text{-Mod}$, there is an isomorphism: $H^\bullet _{BW} (\cC ; M) \cong H_{Th}^\bullet (\cC; \xi^*M)$ (see \cite[Theorem 2.1]{GNT-Thomason-2013}).


\section{Bisimplicial Objects and the Dold-Puppe Theorem}\label{sect:DoldPuppe}

In this section we give some necessary definitions related to bisimplicial objects and homotopy colimits, and state some theorems that we use in the paper.

\subsection{Homotopy Colimits}
Let $\Delta$ denote the simplex category. A \emph{bisimplicial object} in the category $\cC$ is a functor $X: \Delta ^{op}\times \Delta ^{op} \to \cC$.  For every $p, q \geq 0$, we denote the object  $X([p] \times [q])$ by $X_{p,q}$.  The simplicial objects $X_{p, \bullet}$ and $X_{\bullet, q}$ are called vertical and horizontal simplicial objects associated to $X$. The \emph{diagonal of a bisimplicial object} is the simplicial object defined by
\[
\diag X : \Delta ^{op} \to \Delta ^{op} \times \Delta ^{op} \to \cC
\]
where the first functor is the diagonal functor defined by $[n] \to ([n], [n])$. 

A \emph{bisimplicial set} is a bisimplicial object in sets. We define homotopy colimits of simplicial sets using the notation and terminology in \cite{Kahn-2011}. Let $\cD$ be a small category and $F:\cD \to \mathbf{sSet}$ a functor from $\cD$ to the category of simplicial sets. The \emph{simplicial replacement of $F$} is the bisimplicial set $N(\cD ; F)$ 
such that for $p, q\geq 0$,
\[
N(\cD; F)_{p, q} =\coprod _{\sigma = (\sigma(0)\xrightarrow{\sigma_1}  \cdots \xrightarrow{\sigma _p} \sigma(p) ) \in N\cD _p }  F( \sigma(0) )_q 
\]
with vertical and horizontal induced maps defined as follows: 
The vertical face
and degeneracy maps are defined by those of $F(\sigma(0))$. The horizontal face and degeneracy maps are defined such that for every $(\sigma, \tau)\in N(\cD; F)_{p,q}$,
\[
d^h_0 (\sigma ,\tau) = (d_0 \sigma , F(\sigma_1) \tau), \quad d^h_i(\sigma ,\tau) = (d_i \sigma ,\tau) \text{ for } 0 < i \leq p
\]
and $s^h_i (\sigma ,\tau) = (s_i \sigma, \tau)$ for all $0 \leq i \leq p$.

\begin{definition} Let $F: \cD \to \mathbf{sSet}$ be a functor. 
The \emph{homotopy colimit} $\hocolim _{\cD} F$ is the diagonal $\diag N(\cD; F)$ of the simplicial replacement of $F$ (see \cite[\S 12.5.2]{Bousfield-Kan-Book-1987}).
\end{definition}


\subsection{The Dold-Puppe Theorem}

A simplicial object in $R$-modules is called a \emph{simplicial $R$-module}. A functor $A: \Delta ^{op} \times \Delta ^{op} \to R$-Mod is called a \emph{bisimplicial $R$-module}. For a bisimplicial $R$-module $A$, we define the associated Moore double complex $A_{\bullet, \bullet}$ to be the double complex with horizontal and vertical boundary maps
\[
\partial ^h _p =\sum _{i=0} ^p (-1)^i d_i ^h \ \text{ and } \ \partial ^v _q =\sum _{j=0} ^q (-1)^j d_j ^v.
\]
The total complex of the double complex $A_{\bullet, \bullet}$ is the chain complex $\Tot (A_{\bullet, \bullet})$ with
\[
\Tot (A_{\bullet, \bullet}) _n =\bigoplus _{p+q=n} A_{p, q} 
\]
where the boundary maps are given by 
\[
\partial _{n} (x) = \partial ^h _p (x)+ (-1)^p \partial_q ^v (x)
\]
for every $x\in X_{p,q}$ with $p+q=n$.  

Another way to get a chain complex from a bisimplicial $R$-module $A$ is to consider the simplicial $R$-module $\diag A$ and take its Moore complex $(\diag A)_\bullet$.
By the Dold-Puppe theorem, these two different constructions give chain complexes that are homotopic.

\begin{theorem}[Dold-Puppe Theorem \cite{Dold-Puppe-1961}]\label{thm:DoldPuppe}
Let $A: \Delta ^{op} \times \Delta ^{op} \to R$-Mod be a bisimplicial $R$-module. Then there is a natural chain homotopy equivalence $(\diag A)_\bullet \to \Tot(A_{\bullet, \bullet})$. This equivalence is natural with respect to morphisms of bisimplicial $R$-modules.
\end{theorem}

\begin{proof} See \cite[Chp IV, Thm 2.4]{Goerss-Jardine-Book-2009} for a proof.
\end{proof}

The homology of the total complex $\Tot (A_{\bullet, \bullet} )$ can be calculated by spectral sequences associated to the double complex. 
Depending on the chosen filtration, horizontal or vertical, there are two spectral sequences that converge to the homology of the total complex (see \cite{Weibel-Book-1994}): 
\begin{equation}\label{eqn:SpectSeq1} 
\begin{split}
{}^{I} E ^2 _{p, q} &= H^h _p ( H^v _q ( A_{\bullet, \bullet} )) \Rightarrow H_{p+q} (\Tot (A_{\bullet, \bullet})) \\
 {}^{II} E ^2 _{p, q} &= H^v _p ( H^h _q ( A_{\bullet, \bullet} )) \Rightarrow H_{p+q} ( \Tot (A_{\bullet, \bullet}))
 \end{split}
 \end{equation}
These spectral sequences are natural with respect to morphisms between double complexes. 
As a consequence of this naturality, and by applying the Dold-Puppe theorem, we obtain the following.

\begin{proposition}\label{pro:Pointwise} Let $f: A\to B$ be a morphism of bisimplicial $R$-modules. \begin{enumerate}
\item If for every $q\geq 0$, the induced map $f_* : H^h_\bullet (A_{\bullet, q}) \to H^h _\bullet ( B_{\bullet, q} )$ is an isomorphism, then 
$(\diag f)_* : H_\bullet (\diag A) \to H_\bullet (\diag B)$ is an isomorphism. 
\item If for every $p\geq 0$, the induced map $f_* : H^v _\bullet (A_{p, \bullet} ) \to H^v _\bullet ( B_{p, \bullet} )$ is an isomorphism, then 
$(\diag f)_* : H_\bullet (\diag A) \to H_\bullet (\diag B)$ is an isomorphism.
\end{enumerate}
\end{proposition}

\begin{proof} This follows from the spectral sequences above and from Theorem \ref{thm:DoldPuppe}.
\end{proof}


\section{Proof of Theorem \ref{thm:Main1}}\label{sect:Proofmain}

To prove Theorem \ref{thm:Main1}, we will apply an argument used by Cegarra in \cite{Cegarra-2020}. We first recall the terminology used in the statement of Theorem \ref{thm:Main1}.

Let $\varphi : \cC \to \cD$ be a functor between two small categories, and $R$ be a commutative ring. 
The functor $\varphi$ gives rise to a diagram of categories
$\varphi / {-} \colon \cD \to \mathbf{Cat}$ where, for every object $d$ of $\cD$, the
category $\varphi/d$ has objects the pairs $(c,\mu \colon \varphi(c) \to d)$.
A morphism
\[
\alpha \colon (c,\mu) \to (c',\mu ')
\]
in $\varphi/d$ is a morphism $\alpha \colon c \to c'$ in $\cC$ such that
$\mu ' \circ \varphi(\alpha)=\mu$.

If $\beta \colon d \to d'$ is a morphism in $\cD$, then the functor
$\beta _\ast = \varphi/ \beta \colon \varphi/d \to \varphi/d'$ is defined on objects by $\beta_* (c, \mu)= (c, \beta \mu)$, and on morphisms by
\[
\beta_* \bigl ((c, \mu) \xrightarrow{\alpha} (c',\mu') \bigr )=
\bigl ( (c,\beta \mu ) \xrightarrow{\alpha} (c', \beta \mu') \bigr ).
\]

Composing $\varphi / {-}$ with the nerve functor gives a diagram
$N(\varphi/{-}) \colon \cD \to \mathbf{sSet}$. Let
\[
T(\varphi) = \coprod_{\sigma \in N\cD} N(\varphi/\sigma(0))
\]
be its simplicial replacement. The $(p,q)$-bisimplices of $T(\varphi)$ can be identified with  triples $(\sigma,\tau, \mu)$ where
$\sigma \in N\cD _p$, $\tau \in N\cC_q$, and $\mu \colon \varphi(\tau(q)) \to \sigma(0)$
is a morphism in $\cD$. The vertical face
and degeneracy maps are defined on a $(p,q)$-simplex $(\sigma,\tau,\mu)$ by those of $N(\varphi / \sigma (0))$, that is, $d^v _j (\sigma, \tau, \mu) =
 (\sigma, d_j\tau,\mu)$ for $0 \leq j < q$, $d^v_q (\sigma , \tau , \mu)=
(\sigma ,d_q \tau, \mu \varphi (\tau _q ))$, and $s^v_j(\sigma, \tau, \mu)=
(\sigma ,s_j \tau ,\mu)$ for all $0 \leq j \leq q $. Similarly, the horizontal maps are defined by those of
$N\cD$, that is, $d^h_0 (\sigma ,\tau ,\mu) = (d_0 \sigma , \tau , \sigma _1  \mu)$, $d^h_i(\sigma ,\tau ,\mu ) = (d_i \sigma ,\tau ,\mu)$ for $0 < i \leq p$, and $s^h_i (\sigma ,\tau ,\mu ) = (s_i \sigma, \tau,\mu)$ for all $0 \leq i \leq p$.
By definition,
\[
\mathrm{hocolim}_{\cD} N(\varphi/{-}) = \operatorname{diag} T(\varphi).
\]

\begin{lemma} Regarding $N\cC$ as a bisimplicial set constant in the horizontal direction, the map
\[
\kappa \colon T(\varphi) \to N\cC,
\]
defined by 
$\kappa(\sigma,\tau,\mu) = \tau$
is a map of bisimplicial sets and it induces a simplicial map
\[
\kappa \colon \mathrm{hocolim}_{\cD} N(\varphi/{-}) \to N\cC.
\]
\end{lemma}

\begin{proof}
This follows from the face and degeneracy maps for $T(\varphi)$ defined above. Note that since $N\cC$ is constant in the horizontal direction as a bisimplicial set, $\diag N\cC$ is isomorphic to $N\cC$ as a simplicial set.
\end{proof}
 
To prove Theorem \ref{thm:Main1}, we consider the bisimplicial projective 
$R\Delta(N\cC)$-module
\[
\bP=R\coprod_{(\sigma,\tau,\mu)\in T(\varphi)}\mathrm{Hom}_{\Delta(N\cC)}(\tau,-)
\;=\;
\bigoplus_{(\sigma,\tau,\mu)\in T(\varphi)} R\,\mathrm{Hom}_{\Delta(N\cC)}(\tau,-),
\]
and the simplicial projective $R\Delta(NC)$-module
\[
\bQ=R\coprod_{\tau\in N\cC}\mathrm{Hom}_{\Delta(N\cC)}(\tau,-)
\;=\;
\bigoplus_{\tau\in N\cC} R\,\mathrm{Hom}_{\Delta(N\cC)}(\tau,-),
\]
Let $K: \bP \to \bQ$ be the bisimplicial homomorphism induced by $\kappa :T(\varphi)\to N\cC $,
where $\bQ$ is considered constant in the horizontal direction. We also denote 
the simplicial homomorphism induced on diagonals
by $K: \operatorname{diag}\bP\to \bQ$.

\begin{lemma}\label{lem:IsomComplexes} The simplicial homomorphism $K:\operatorname{diag}\bP\to \bQ$ induces isomorphism on homology.
\end{lemma}

\begin{proof} For every $\theta \in N\cC _n$, $\bP_{p,q} (\theta)$ is the free $R$-module with basis given by the set
\[
X_{p,q} (\theta) = \{ (\sigma, \tau,\mu, f) \, | \, (\sigma, \tau, \mu) \in T(\varphi)_{p.q}, f: [q]\to [n],  f^* (\theta)=\tau \}.
\]
Similarly, for every $\theta \in N\cC_n$, $\bQ_q (\theta)$ is a free $R$-module with basis given by 
\[
Y_q(\theta)=\{ (\tau, f) \, |\, \tau \in N\cC_q, f: [q]\to [n] \text{ such that } f^* (\theta )=\tau \}.
\]
For every $\theta \in N\cC_n $, the simplicial homomorphism $K_{\bullet, q} : \bP _{\bullet, q} (\theta) \to \bQ _q (\theta)$ sends the basis element $(\sigma, \tau, \mu, f)$ to $(\tau, f)$. The simplicial homomorphism
$K_{\bullet, q}$ can be described as the map
\[
\bP_{\bullet,q}(\theta)=
R\coprod_{f\in (\Delta^n)_q} N(\varphi \theta f(q) /\cD)
\longrightarrow
R\coprod_{f\in (\Delta^n)_q} * \;=\; R(\Delta^n)_q=\bQ_q(\theta),
\]
where $\Delta^n=\mathrm{Hom}_\Delta(-,[n])$ is the standard simplicial $n$-simplex, and $\varphi \theta f(q) /\cD$ denotes the right comma category for the object $\varphi \theta f(q)$ in $\cD$. For every $f\in (\Delta ^n )_q$, the simplicial map 
$R\,N(\varphi \theta f(q) /\cD)\to R*\cong R$ is a homology isomorphism
since for every object $d$ of $\cD$, the comma category $d/\cD$ has an initial object. 
Hence for every $q\ge 0$ and for every $\theta \in N\cC_n$, the simplicial homomorphism $K_{\bullet, q} (\theta) :\bP_{\bullet,q} (\theta) \to \bQ_q (\theta) $ is a homology isomorphism.
It follows from the Dold--Puppe Theorem
that the induced map on diagonals $\diag K (\theta) :\operatorname{diag}\bP(\theta)\to \bQ(\theta)$ is a homology
isomorphism. Therefore, $K:\operatorname{diag} \bP\to \bQ$ is a homology isomorphism.
\end{proof}

We also observe the following:

\begin{lemma}\label{lem:Contractible}
For every $q>0$, we have $H_q(\bQ)=0$ and $H_0(\bQ)=\underline R$.
Hence $\bQ$ and $\diag \bP$ are both simplicial projective resolutions of $\underline R$ as $R\Delta(N\cC)$-modules and the simplicial homomorphism 
$K :\operatorname{diag}\bP\to \bQ$ is a homotopy equivalence between these simplicial projective resolutions.
\end{lemma}

\begin{proof}
For every
$n$-simplex $\theta:[n]\to \cC$ of $N\cC$, $\bQ(\theta)=R\Delta^n$ which is
contractible since the category $[n]$ has an initial object. Hence $\bQ (\theta)$ has the homology of a point for every $n$-simplex $\theta$ in $N\cC$. The second statement
follows from Lemma \ref{lem:IsomComplexes}.
\end{proof}

Now we are ready to prove Theorem \ref{thm:Main1}.

\begin{proof}[Proof of Theorem \ref{thm:Main1}]
For every given coefficient system $M$ on $N\cC$, let
\begin{equation}\label{eqn:CphiM}
C(\varphi; M)=\prod_{(\sigma,\tau,\mu)\in T(\varphi)} M(\tau)
\end{equation}
be the bicosimplicial $R$-module whose vertical and horizontal face and degeneracy homomorphisms are induced by face and degeneracy maps of $T(\varphi)$.
By definition, we have
\begin{equation}\label{eq:diagC}
\operatorname{diag} C(\varphi;M)=
C \bigl (\operatorname{hocolim}_\cD N(\varphi/{-}); \kappa ^* M \bigr ).
\end{equation}
Regarding $C(\cC;M)$ as a bicosimplicial $R$-module constant in the horizontal direction, let
\[
\kappa ^*:C(\cC;M)=\prod_{\tau\in N\cC} M(\tau)\longrightarrow
\prod_{(\sigma,\tau,\mu)\in T(\varphi)} M(\tau)=C(\varphi;M)
\]
be the bicosimplicial homomorphism induced by the bisimplicial map $\kappa: T(\varphi) \to N\cC$ defined by $\kappa(\sigma, \tau, \mu)=\tau$. Then the induced map on diagonals
\[
\kappa ^*:C(\cC;M)\to \operatorname{diag}C(\varphi;M)
\]
is precisely the map claimed to be a cohomology isomorphism.

To prove this, consider the simplicial projective resolutions $\diag \bP$ and $\bQ$ constructed in Lemma \ref{lem:IsomComplexes}.  There is a commutative diagram of cosimplicial $R$-modules
\[
\xymatrix{\mathrm{Hom}_{\Delta(N\cC)}(\bQ;M) \ar[d]^{\kappa^*} \ar[r]^-{\cong} & C(\cC;M) \ar[d]^{\kappa^*}\\
\mathrm{Hom}_{\Delta(N\cC)}(\diag\bP;M) \ar[r]^-{\cong} & \diag C(\varphi;M)}
\]
where the horizontal isomorphisms are given by the Yoneda Lemma. The induced map $\kappa ^*$ on the left is a cohomology isomorphism because, by Lemma \ref{lem:Contractible}, the simplicial homomorphism $K :\diag \bP \to \bQ$ is a homotopy equivalence between simplicial projective resolutions of $\underline R$.
Hence we conclude that the homomorphism $\kappa^*$ on the right is a cohomology isomorphism.
\end{proof}


\section{A Spectral Sequence for Grothendieck Construction}
\label{sect:SpectSeq}

In this section we construct the spectral sequence for a functor $\varphi: \cC \to \cD$ that converges to the Thomason cohomology of $\cC$. As an application we construct a spectral sequence for the cohomology of the Grothendieck construction of a functor $F : \cC \to \mathbf{Cat}$.

Let $\varphi : \cC \to \cD$ be a functor between two small categories and $M$ be a coefficient system on $N\cC$. For each object $d$ of $\cD$, let
\[
j:\varphi/d\to \cC
\]
be the functor that carries each object $(c,\mu)$ to $c$. 
Let $C(\varphi/-; j^*M)$ denote the functor from $\cD^{op}$ to the category of cosimplicial
$R$-modules, which sends an object $d$ of $\cD$ to $C(\varphi/d; j ^*M)$, and carries a morphism $\beta :d\to d'$ in $\cD$ to the homomorphism
\[
\beta ^*:C(\varphi/d'; j^*M)\to C(\varphi/d; j^*M)
\]
induced by the functor $\beta _*:\varphi/d\to \varphi/d'$. Hence, we can form the bicosimplicial $R$-module
\[
C\bigl(\cD^{op}; C(\varphi/-; j^*M)\bigr)
=
\prod_{\sigma\in N\cD} C(\varphi/\sigma(0); j^* M)
\]
and observe the following:

\begin{lemma} Let $C(\varphi; M)$ be the bicosimplicial
$R$-module defined in (\ref{eqn:CphiM}).  
Then
\[
C(\varphi;M) = C\bigl(\cD^{op}; C(\varphi/-; j^*M)\bigr)
\]
as bicosimplicial $R$-modules.
\end{lemma}

\begin{proof}
This follows by the definition of the bicosimplicial $R$-module $C(\varphi; M)$.
\end{proof}

For each integer $q\ge 0$, let
\[
H^q_{Th}(\varphi/-; j^*M):\cD ^{op}\to R\text{-}\mathrm{Mod}
\]
be the functor that assigns to every object $d$ of $D$, the cohomology $R$-module $H^q_{Th}(\varphi/d, j^*M)$. We prove the following theorem.

\begin{theorem}\label{thm:Main2-SpectSeq}
For every coefficient system $M \colon \Delta(N\cC) \to R\text{-Mod}$, there is a
first-quadrant spectral sequence
\[
E^{p,q}_2
=
H^p\bigl(\cD^{op}; H^q_{Th}(\varphi/{-}; j^\ast M)\bigr)
\;\Longrightarrow\;
H^{p+q}_{Th}(\cC; M)
\]
where for each object $d$ of $\cD$, $j: \varphi/d \to \cC$ is the forgetful functor that sends each object $(c, \mu)$ to $c$.
\end{theorem}

\begin{proof}
By (\ref{eqn:SpectSeq1}), there is a spectral sequence of the form 
\[
{}^{I} E _2 ^{p, q} = H_h ^p ( H_v ^q C(\varphi; M) )\Rightarrow H^{p+q} (\Tot C ^{\bullet, \bullet}(\varphi; M)). 
\]
Fixing any $q\ge 0$ and taking vertical $q$-cohomology in $C(\varphi;M)$ gives
\[
H^q_v C(\varphi;M)\cong C\bigl(\cD^{op}; H^q_{Th}(\varphi/-; j^*M)\bigr).
\]
Taking now the $p$-cohomology of $H_v ^q C(\varphi, M)$ gives
\[
H^p_h H^q_v C(\varphi;M)\cong H^p\bigl(\cD^{op}; H^q_{Th}(\varphi/-; j^*M)\bigr).
\]
By the Dold-Puppe Theorem, $H^{p+q} (\Tot C ^{\bullet, \bullet}(\varphi; M) ) \cong H^* (\diag C(\varphi; M))$.
By \eqref{eq:diagC} and Theorem \ref{thm:Main1}, we have $H^\bullet (\mathrm{diag}\,C(\varphi;M) ) \cong H^\bullet_{Th}(\cC;M)$. Hence
the claimed spectral sequence follows.
\end{proof}

For a particular instance of the spectral sequence above, we obtain a spectral sequence for the Grothendieck construction of a functor $F:\cD \to \mathbf{Cat}$ stated in the introduction as Theorem \ref{thm:GrothSpectSeq}.

\begin{proof}[Proof of Theorem \ref{thm:GrothSpectSeq}]
Suppose $F:\cD\to \mathbf{Cat}$ is
a functor and let $\int_{\cD} F$ be the category obtained by the Grothendieck construction on $F$.
Its objects are pairs $(d,x)$ where $d$ is an object of $\cD$ and $x$ is an object of $F(d)$.
A morphism $(\alpha,\gamma):(d,x)\to (d',x')$ in $\int_{\cD} F$ consists of a morphism $\alpha :d\to d'$ in $\cD$
together with a morphism $\gamma: F(\alpha)(x)\to x'$ in $F(d')$. Let
\[
\pi:\int_{\cD} F\to \cD
\]
denote the forgetful functor that sends each object $(d,x)$ to $d$. For every object $d$ of $\cD$, $j : \pi/d \to \int_{\cD} F$ denotes the functor
that sends an object $((d',x'), \mu)$ in $\pi/d$ to $(d', x')$ in $\int_{\cD} F$. By applying Theorem \ref{thm:Main2-SpectSeq} to $\pi: \int_{\cD} F \to \cD$, we obtain the spectral sequence
in the statement of Theorem \ref{thm:GrothSpectSeq}.
\end{proof}

As a special case of the spectral sequence in Theorem \ref{thm:GrothSpectSeq}, consider the following situation:
Let $G$ be a group and $X$ be a $G$-simplicial set. In this case 
the simplex category $\Delta (X)$ is a $G$-category. Let $\cG$ denote the one object group category whose morphisms are given by the group $G$, and let  $F : \cG \to \mathbf{Cat}$ be the functor that sends the object $\ast$ in $\cG$ to $\Delta (X)$. We denote the Grothendieck construction $\int _{\cG} F$ by $\Delta (X)_G$.

The objects of $\Delta(X)_G$ are the pairs $(\ast, \sigma)$ where $\sigma\in \Delta(X)$. A morphism $(\ast, \tau) \to (\ast, \sigma)$, where $\sigma\in X_n$ and $\tau\in X_m$, is given by a pair $(g, f)$ where $g\in G$ and $f: [m]\to [n]$ is a morphism in the simplex category $\Delta$ such that $f^* (\sigma)= g \tau$. A functor $M : \Delta (X)_G\to R$-Mod is called a \emph{$G$-equivariant coefficient system on $X$}.  

Consider the forgetful functor $\pi: \Delta (X)_G \to \cG$ that sends every object $(\ast, \sigma)$ to $\ast$, and every morphism $(g, f)$ to $g\in G$.
If we apply Theorem \ref{thm:GrothSpectSeq} to the category $\Delta (X)_G$, we obtain the following corollary.

\begin{corollary}\label{cor:Main3-GenSerreSpect}
Let $X$ be a $G$-simplicial set and $M$ an equivariant coefficient system on $X$. Then there is a spectral sequence
\[
E_2 ^{p,q} =H^p ( G ; H ^q _{M} ) \Rightarrow H^{p+q} _{Th} ( \Delta(X)_G; M )
\]
where $H^p (G; H_M ^q)$ denotes the cohomology of the group $G$ with coefficients in the $RG$-module $H_M ^q := H^q _{Th} ( \pi /*; j^* M )$.
\end{corollary}

\begin{remark} As a consequence of Thomason's homotopy colimit theorem, the nerve of the category $\Delta (X)_G$ is homotopy equivalent to 
\[
\hocolim _{\cG} N\Delta (X)
\]
which is homotopy equivalent to the Borel construction $EG \times_G X$ (see \cite[Lemma 2.3]{Grodal-Endo}). This means that when the coefficient system $M$ is the constant coefficient system $\underline R$, the above spectral sequence becomes the usual Serre spectral sequence
\[
E_2 ^{p,q} =H^p(G ; H^q (X; R)) \Rightarrow H^{p+q}  ( EG\times _G X ; R).
\]
In this sense the above spectral sequence can be considered as the generalization of the Serre spectral sequence to the cohomology of $\Delta (X)_G$ with arbitrary $G$-equivariant coefficient systems.
\end{remark}


\section{Quillen's Theorem A for Thomason Cohomology} 

In this section we prove Theorem \ref{thm:Main2-QuillenA} stated in the introduction. For every object $d$ of $D$, let 
\[
u:\varphi/d\to \mathrm{id}_\cD/d = \cD/d
\]
be the forgetful functor that carries each object $(c,\mu)$ to $(\varphi(c), \mu)$. The square
\[
\xymatrix{ \varphi/d \ar[r]^{j} \ar[d]_{u}  & \cC \ar[d]^{\varphi} \\
\cD/d \ar[r]^{j} & \cD.}
\]
commutes and, for each coefficient system $M:\Delta(N\cD)\to R\text{-}\mathrm{Mod}$, the functor $u$ induces
a cosimplicial homomorphism
\begin{equation}\label{eq:u-star}
u^*:C(\cD/d; j^*M)\to C(\varphi/d; u^*j^*M)=C(\varphi/d;j^*\varphi^*M),
\end{equation}
and therefore a homomorphism
\begin{equation}\label{eq:u-star-cohom}
u^*:H^\bullet_{Th}(\cD/d; j^*M)\to H^\bullet_{Th}(\varphi/d;j^*\varphi^*M).
\end{equation}

\begin{proof}[Proof of Theorem \ref{thm:Main2-QuillenA}]
The cosimplicial homomorphisms $u^*$ in \eqref{eq:u-star} are natural in the objects $d$ of $\cD$. Define a
bicosimplicial homomorphism
\[
C(\mathrm{id}_\cD;M)
=
C\bigl(\cD^{op};C(\cD/-;j^*M)\bigr)
\;\xrightarrow{\ u^*\ }\;
C\bigl(\cD^{op};C(\varphi/-;j^*\varphi^*M)\bigr)
=
C(\varphi;j^*\varphi^*M).
\]
Since by hypothesis, at every $p\ge 0$, the cosimplicial homomorphism
\[
\prod_{\sigma\in N\cD_p} C(\cD/\sigma(0);j^*M)
\;\xrightarrow{\ u^*\ }\;
\prod_{\sigma\in N\cD_p} C(\varphi/\sigma(0);j^*\varphi^*M)
\]
is a cohomology isomorphism, it follows from
Proposition \ref{pro:Pointwise} that the induced map on diagonals
\[
\mathrm{diag}\,C(\mathrm{id}_\cD;M)\xrightarrow{\ u^*\ }\mathrm{diag}\,C(\varphi;\,j^*\varphi^*M)
\]
is also a cohomology isomorphism. Hence, by \eqref{eq:diagC},
\[
C\bigl(\mathrm{hocolim}_\cD N(\cD/-); j^*M\bigr)
\;\xrightarrow{\ u^*\ }\;
C\bigl(\mathrm{hocolim}_\cD N(\varphi/-); j^*\varphi^*M\bigr)
\]
is a cohomology isomorphism. Now the result follows from the commutativity of the square of cosimplicial
homomorphisms below, where the vertical homomorphisms $\kappa ^*$ are cohomology isomorphisms by Theorem \ref{thm:Main1}:
\[
\xymatrix{C(\cD;M) \ar[r]^{\varphi^*} \ar[d]^{\kappa^*} & C(\cC;\varphi^*M) \ar[d]^{\kappa^*} \\
C\bigl(\mathrm{hocolim}_\cD N(\cD/-); j^*M\bigr) \ar[r]^-{u^*} & C\bigl(\mathrm{hocolim} _{\cD}N(\varphi/-); j^*\varphi^* M \bigr).}
\]
\end{proof}

Quillen's Theorem A states that if $\varphi: \cC \to \cD$ is a functor such that for every object $d$ of $\cD$, the comma category $\varphi/d$ is contractible, then
$N\varphi : N\cC\to N\cD$ is a homotopy equivalence. When $\varphi /d$ is contractible, the simplicial map $Nu: N(\varphi/d) \to N(\id_{\cD}/d )$ is a homotopy equivalence, but in general it may not 
induce an isomorphism on cohomology with arbitrary coefficients. This means that  under the hypothesis of Quillen's Theorem A, in general we cannot conclude that the induced map $\varphi ^*$ on cohomology is an isomorphism.

The following example due to Husainov \cite{Husainov-2023} shows that
under the hypothesis of Quillen's Theorem A, the induced map  $\varphi ^*: H_{Th}^\bullet (\cD; M) \to H_{Th}^\bullet ( \cC; \varphi^* M )$ is not an isomorphism for some coefficient systems.

\begin{example} Suppose $\cC$ is the category with one object $\{0\}$ and with only the identity morphism $\id_0$, and $\cD$ is the category with objects $\{0, 1\}$ and with one non-identity morphism $\alpha: 0\to 1$. Let $\varphi : \cC \to \cD$ be the inclusion functor. Then for every object $d$ of $\cD$, the comma category $\varphi/d$ is contractible. Consider the functor between the factorization categories $F\varphi: F\cC \to F\cD$.  
Let $M$ be the natural system for $\cD$ over $\bbZ$ with values $M(\id _0)=M(\id _1)=0$ and $M(\alpha)\cong\bbZ$. 
It is easy to see that $H^1 _{BW} ( \cC; M)=0$ and $H^1 _{BW} (\cD; M)\cong \bbZ$, so the homomorphism induced by $\varphi$ on Baues-Wirsching cohomology is not an isomorphism (see \cite[Example 5.3]{Husainov-2023} for details). Hence we can conclude that the induced map on Thomason cohomology $\varphi ^* : H^\bullet_{Th} (\cD; M) \to H^\bullet _{Th} (\cC; \varphi^* M)$ is not an isomorphism for some coefficient systems $M$ on $N\cD$.
\end{example}

\begin{remark} In \cite{Husainov-2023}, Husainov gives  counterexamples to some statements that appear in the literature. Husainov's examples show that for an adjoint pair $(l, r)$ such that $l: \cC \rightleftarrows \cD : r$, the induced map 
\[
l^*: H^*_{Th} (\cD; M) \to H^*_{Th} (\cC; l^* M )
\]
is not always an isomorphism. His examples show in particular that \cite[Lem 2.2]{Pirashvili-Redondo-2006} and \cite[Prop 1.11 and Cor 2.3]{GNT-Thomason-2013} do not hold in the way they are stated.  
\end{remark}


\section*{Acknowledgements and Funding Declaration}
The first author is supported by T\" UB\. ITAK 2211-A National PhD Scholarship Program. The second author is supported by T\" UB\. ITAK 2219-International Postdoctoral Research Fellowship Program (2023, 2nd term). We gratefully acknowledge T\" UB\. ITAK for its support of this research. 

We thank the referee for a careful reading of the paper, for pointing out a mistake in an earlier version, and for detailed suggestions that significantly improved the paper. We also thank Matthew Gelvin and Caroline Yal\c c\i n for reading an earlier version of the paper and for their helpful comments.

\section*{Author Contributions Statement}
All authors contributed to the study conception and design.  The first draft of the manuscript was written by both authors and all authors commented on previous versions of the manuscript. All authors read and approved the final manuscript.

\section*{Data Availability}
 This manuscript does not report data generation or analysis.
 
\section*{Conflict of Interest}

The authors declare that they have no conflicts of interest.

\section*{Ethics Declaration}

Not applicable.


\end{document}